\documentclass[amsfonts]{article}
\usepackage{amsmath}
\usepackage{amsfonts}
\usepackage{amsthm}
\usepackage{amssymb}
\usepackage{enumerate}
\theoremstyle{plain}
\newtheorem{theo}{Theorem}[section]
\newtheorem{prop}[theo]{Proposition}
\newtheorem{definition}[theo]{Definition}
\theoremstyle{remark}
\newtheorem{rem}[theo]{Remark}

\newcommand{\RR}{{\mathbb{R}}}
\newcommand{\CC}{{\mathbb C}}

\newcommand{\II}{{\mathbb I}}

\newcommand{\der}{\partial}

\newcommand{\ov}{\overline}

\newcommand{\om} {\Omega}

\newcommand{\nabhat}{{\widehat{\nabla}}}
\newcommand{\R}{\mathbb{R}}

\begin{document}

\title{Lipschitz continuous dependence of piecewise constant Lam\'e coefficients {}from boundary data: the case of non flat interfaces
 }
\author{Elena Beretta\thanks{Politecnico di Milano, Italy
(elena.beretta@polimi.it)} \and Elisa Francini\thanks{Universit\`{a} degli Studi di Firenze, Italy (elisa.francini@unifi.it)}\and
Antonino Morassi\thanks{
Universit\`{a} degli Studi di Udine, Italy (antonino.morassi@uniud.it)}\and
Edi Rosset\thanks{
Universit\`{a} degli Studi di Trieste, Italy (rossedi@units.it)}\and
Sergio Vessella\thanks{
Universit\`{a} degli Studi di Firenze, Italy (sergio.vessella@unifi.it)}}
\date{}
\maketitle
\begin{abstract}

We consider the inverse problem of determining the Lam\'e moduli for a piecewise constant elasticity tensor
$\CC= \sum_{j} \CC_j \chi_{D_j}$, where $\{D_j\}$ is a known finite partition of the body $\Omega$, {}from the Dirichlet-to-Neumann map.
We prove that Lipschitz stability estimates can be derived under $C^{1,\alpha}$ regularity assumptions on the interfaces.

\end{abstract}


\noindent
\textbf{keywords: }{Inverse boundary value problem, Lam\'{e} system, piecewise constant coefficients, Lipschitz stability}

\section{Introduction\label{sec0}}

An important inverse problem arising {}from engineering
sciences consists in determining the elasticity coefficients
of the material occupying a three dimensional body {}from
measurements of tractions and displacements taken on its
accessible boundary.

The boundary value problem {}from which this inverse problem
originates is as follows. Let $\Omega$ be a bounded domain in
$\RR^3$ made by a linearly elastic isotropic material, with Lam\'{e}
moduli $\mu$ and $\lambda$ satisfying the strong
convexity conditions $\mu(x)\geq \alpha_0 >0$, $2\mu(x) +
3\lambda(x) \geq \beta_0$ in $\Omega$, for some positive constants
$\alpha_0$ and $\beta_0$. For a given $\psi \in H^{
\frac{1}{2}}(\partial \Omega)$, the direct problem consists in
finding the displacement field $u \in H^1(\Omega)$ solution to the
Dirichlet problem
\begin{equation}
    \label{I1.1}
    \left\{\begin{array}{rcl}
             \mbox{div}(\CC  \nabhat u)& = & 0, \quad \mbox{ in }\om, \\
             u & = & \psi, \quad \mbox{ on }\der\om,
           \end{array}\right.
\end{equation}
where $\CC= \lambda(x) I_3 \otimes I_3 + 2\mu(x){\II}_{Sym}$ is
the Lam\'{e} elasticity tensor.

We denote by $\Lambda_{\CC}: H^{ \frac{1}{2}}(\partial \Omega)
\rightarrow H^{ - \frac{1}{2}}(\partial \Omega)$ the
Dirichlet-to-Neumann map associated to the problem \eqref{I1.1},
that is the operator which maps the Dirichlet data $u|_{\partial
\Omega} = \psi$ onto the corresponding Neumann data $(\CC \nabhat
u) \nu|_{\partial \Omega}$, where $\nu$ is the outer unit normal
to $\Omega$.

An interesting inverse problem is the determination of the Lam\'{e} coefficients $\mu$ and $\lambda$ when
$\Lambda_{\CC}$ is known. Most of the results available in the
literature concern the uniqueness issue. A linearized version of
this inverse problem was considered by Ikehata \cite{Ik}. In
\cite{NU1}, Nakamura and Uhlmann established that in two
dimensions the Lam\'{e} moduli are uniquely determined by
$\Lambda_{\CC}$, provided that they are smooth (e.g., $C^\infty(
\overline{\Omega}))$ and sufficiently close to positive constants.
The uniqueness in dimension three was proved in \cite{NU2},
\cite{ER}, \cite{NU3}, assuming that the Lam\'{e} moduli are
$C^\infty( \overline{\Omega})$ and $\mu$ close to a positive
constant. Recent results concern the uniqueness in the case of
partial Cauchy data, see \cite{IUY} for details.

The stability issue for the above inverse problem is expected to
be significantly more difficult than uniqueness and, to our
knowledge, no general result is known. In the simpler context of
an electric conductor, which involves the determination of a
single smooth coefficient in a scalar elliptic equation {}from
boundary measurements, it is well-known that the optimal rate of
continuous dependence is of logarithmic type, see, for instance, \cite{Al1} and \cite{M}. It
follows that logarithmic stability estimates, or even worse ones,
are expected in our case. In addition, the
situation is more complicated because, in several practical
applications, the Lam\'{e} moduli are not smooth and, in some
cases, may also be discontinuous.

In order to have better stability results, a possible way
is based on the
introduction of \emph{suitable} a priori assumptions that are physically
relevant and restore well-posedness. Following the approach suggested by
Alessandrini and Vessella \cite{AV} in the conductivity framework,
in \cite{BFV} the
authors considered a class of piecewise constant elasticity
tensors of the form
\begin{equation}
    \label{I4.1}
    \CC=\sum_{j=1}^N (\lambda_j I_3\otimes I_3
    +2\mu_j{\II}_{Sym})\chi_{D_j}(x),
\end{equation}
where the collection of disjoint Lipschitz domains
$\{D_j\}_{j=1}^N$ forms a known decomposition of the domain
$\Omega$, and $\lambda_j$, $\mu_j$, $j=1,...,N$, are unknown
constants to be determined {}from $\Lambda_\CC$. Assuming that the boundaries of the domains $D_j$
contain \textit{flat portions}, the authors were able to prove a
Lipschitz continuous dependence of the Lam\'{e} moduli {}from the
local Dirichlet-to-Neumann map.

The structure \eqref{I4.1} assumed for $\CC$ fits well in several
problems arising in applications. Polyhedral partitions of
$\Omega$ appear frequently in finite element meshing used for
effective reconstruction of the Lam\'{e} parameters \cite{BJK}. In
identification of material properties of masonry walls or concrete
dams, for example, the actual elasticity coefficients are
approximated by assuming that each finite element or group of
finite elements is made by homogeneous Lam\'{e} material. The
partition of the domain is often suggested by a priori information
on different grades of the material or, in the case of dams, by
the possible presence of natural joints inside the concrete \cite{XJY}.
Obviously, it is not always possible to ensure that the domains
$D_j$ have a flat portion of their boundary in common and,
therefore, in order to address these more general inverse
problems, it is necessary to remove this a priori assumption.

In this paper we prove
a Lipschitz stability estimate assuming
$C^{1,\alpha}$ regularity of some portions $\Sigma_j$ of the interfaces joining contiguous
domains $D_{j-1}$, $D_j$ and on the portion $\Sigma$ of $\partial \Omega$
where the measurements are taken. The precise regularity conditions are
given in Section \ref{subsec2.3} (assumptions \textbf{(A1)}).

Our proof is inspired by the paper \cite{BFV} and is mainly based
on the use of unique continuation properties and on a refined local analysis, near the
$C^{1,\alpha}$ interface $\Sigma_j$, of the behavior of the corresponding biphase fundamental solution
(see Subsection \ref{subsec-sing} for the precise setting). To this aim, a new mathematical tool is
the recent asymptotic approximation of this fundamental solution (see \cite{AdCMR}) in terms
of the biphase fundamental solution associated to a flat interface, which was determined in a close form
by Rongved \cite{R}.

We follow a slightly different procedure to prove the stability estimate. In \cite{BFV} the authors reformulate the direct problem
in terms of the nonlinear forward map $F$ acting on a compact subset $K$ of
$\RR^{2N}$, and use an abstract lemma (see \cite{BaVe}) which ensures that the inverse map
$\left(F|_K\right)^{-1}$ is Lipschitz continuous. Here, instead, we give a more direct proof following the lines
in \cite{AV} for the conductivity framework. As in \cite{AV}, \cite{BFV}, also our proof proceeds
by induction. However, in order to simplify the presentation and to emphasize the crucial points
where new tools are needed, we focus on the first two steps of the induction process. Precisely, the key role
of the asymptotic approximation of the biphase fundamental solution is emphasized in the first step, whereas
the second step explains how to use the transmission conditions at the interface and the stability estimates for
the Cauchy problem to propagate the smallness crossing an interface.

\section{Main result}\label{sec2}
\subsection{Notation and main definitions}\label{subsec2.1}
For every $x\in\RR^3$ we set  $x=(x^\prime,x_3)$ where
$x^\prime\in\RR^{2}$ and $x_3\in \RR$. For every $x\in \RR^3$, $r$
and $L$ positive real numbers we will denote by $B_r(x)$,
$B_r^\prime(x^\prime)$ and $Q_{r,L}(x)$  the open ball in $\RR^3$
centered at $x$ of radius $r$, the open ball in $\RR^2$ centered
at $x^\prime$ of radius $r$ and the cylinder
$B_r^\prime(x^\prime)\times (x_3-Lr,x_3+Lr)$, respectively. In the
sequel $B_r(0)$, $B_r^\prime(0)$ and $Q_{r,L}(0)$ will be denoted
by $B_r$, $B^\prime_r$ and $Q_{r,L}$, respectively.
 We will also denote by $\RR^3_+=\{(x^\prime,x_3)\in\RR^3 | x_3>0\}$,
$\RR^3_-=\{(x^\prime,x_3)\in \RR^3 | x_3<0\}$, $B^+_r=B_r\cap
\RR^3_+$, and $B^-_r=B_r\cap \RR^3_-$.

For any subset $D$ of $\RR^3$ and any $h>0$, we denote by
\[(D)_{h}=\{x \in D| \mbox{ dist}(x,\RR^3\setminus D)>h\}.\]

\begin{definition}
  \label{def:2.1} (${C}^{k,\alpha}$ regularity)
Let $U$ be a bounded domain in $\mathbb{R}^{3}$. Given $k$,
$\alpha$, with $k\in \mathbb{N}$ and $0<\alpha\leq 1$, we say that
$U$ is of class ${C}^{k,\alpha}$ with constants $r_{0}$, $L$, if,
for any $P \in \partial U$, there exists a rigid transformation of
coordinates under which we have $P=0$ and
\begin{equation*}
  U \cap B_{r_{0}}(0)=\{x \in B_{r_{0}}(0) | x_{3}>\varphi(x')
  \},
\end{equation*}
where $\varphi$ is a ${C}^{k,\alpha}$ function on $\RR^2$
satisfying
\begin{equation*}
\varphi(0)=0,
\end{equation*}
\begin{equation*}
\nabla \varphi (0)=0, \quad \hbox{when } k \geq 1,
\end{equation*}
\begin{equation*}
\|\varphi\|_{C^{k,\alpha}(\RR^2)} \leq Lr_{0}.
\end{equation*}

\medskip
\noindent When $k=0$, $\alpha=1$, we also say that $U$ is of
\textit{Lipschitz class with constants $r_{0}$, $L$}.
\end{definition}

\begin{rem}
  \label{rem:2.1}
  We use the convention to normalize all norms in such a way that their
  terms are dimensionally homogeneous and coincide with the
  standard definition when the dimensional parameter equals one.
  For instance, the norm appearing above is meant as follows
\begin{equation*}
  \|\varphi\|_{{C}^{k,\alpha}(\RR^2)} =
  \sum_{i=0}^k r_0^i
  \|\nabla^i\varphi\|_{{L}^{\infty}(\RR^2)}+
  r_0^{k+\alpha}|\nabla^k\varphi|_{\alpha, \RR^2},
\end{equation*}
where
\begin{equation*}
|\nabla^k\varphi|_{\alpha, \RR^2}= \sup_ {\overset{\scriptstyle
x', y'\in \RR^2}{\scriptstyle x'\neq y'}}
\frac{|\nabla^k\varphi(x')-\nabla^k\varphi(y')|} {|x'-y'|^\alpha}.
\end{equation*}

Similarly,
\begin{equation*}
\|u\|_{H^m(\Omega)}=\left(\sum_{i=0}^m
r_0^{2i-3}\int_\Omega|\nabla^i u|^2\right)^{\frac{1}{2}},
 \quad
  \|u\|_{{C}^{k}(\Omega)} =\sum_{i=0}^{k}
  {r_{0}}^{i}\|\nabla^{i} u\|_{{L}^{\infty}(\Omega)},
\end{equation*}
\begin{equation*}
\|u\|_{L^2(\partial\Omega)}=\left(
r_0^{-2}\int_{\partial\Omega}|u|^2\right)^{\frac{1}{2}},
\end{equation*}
where $H^0(\Omega) = L^2(\Omega)$, and so on for trace norms such as
$\|\cdot\|_{H^{\frac{1}{2}}(\partial \Omega)}$,
$\|\cdot\|_{H^{-\frac{1}{2}}(\partial \Omega)}$, where $\Omega$ is
a bounded subset of $\RR^3$ with regular boundary.
\end{rem}

We will also make use of the following notation for matrices and
tensors. Let $\mathcal{L}(\RR^3,\RR^3)$ be the linear space of $3
\times 3$ matrices. For any $A, B \in \mathcal{L}(\RR^3,\RR^3)$,
we set $A : B=\sum_{i,j=1}^3 A_{ij}B_{ij}$, $|A|^2 = A : A$
 and $\hat A=\frac{1}{2}(A+A^T)$. By $I_3$ we denote the $3\times 3$ identity matrix and by ${\II}_{Sym}$ we denote the fourth order tensor such that ${\II}_{Sym}A=\hat A$.

\bigskip

In the whole paper we are going to consider isotropic elastic
materials, hence the fourth order elasticity tensor $\CC$ is given
by
\begin{equation}\label{isotr}
\CC(x)=\lambda(x)I_3\otimes I_3 +2\mu(x){\II}_{Sym}, \quad\hbox{
for  a.e.  }x\mbox{ in } \Omega,
\end{equation}
where $\Omega$ is a bounded domain in $\RR^3$ of Lipschitz class and $(I_3\otimes I_3)A= (I_3 : A) A$ for every
$3 \times 3$ matrix $A$.
The real valued functions $\lambda=\lambda(x)$ and $\mu=\mu(x)\in
L^{\infty}(\om)$ are the Lam\'{e} moduli, and satisfy the strong
convexity condition
\begin{equation}
\label{strong-convexity}
     \alpha_0\leq \mu(x) \leq \alpha_0^{-1}, \quad \lambda(x) \leq
    \alpha_0^{-1}, \quad 2\mu(x) +3\lambda(x) \geq \beta_0, \quad
    \hbox{for a.e. } x \hbox{ in } \Omega,
\end{equation}
where $\alpha_0\in (0,1]$, $\beta_0 \in(0,2]$ are given constants. Let us
notice that the Poisson's ratio $\nu(x) =
\frac{\lambda(x)}{2(\lambda(x)+\mu(x))}$ satisfies
\begin{equation}
\label{Poisson}
    -1 + \frac{\alpha_0 \beta_0}{4} \leq \nu(x) \leq \frac{1}{2} -
    \frac{\alpha_0^2}{4}, \quad
    \hbox{for a.e. } x \hbox{ in } \Omega.
\end{equation}
Under these assumptions, the elasticity tensor  $\CC$ satisfies
the minor and major symmetry conditions
\begin{equation}
\label{symmetry-of-C}
    \CC  A = \widehat{A}, \quad \CC A : B= \CC B : A,
\end{equation}
and the strong convexity condition
\begin{equation}
\label{strong-convexity-of-C}
    \CC  A : A \geq \xi_0 |A|^2,
\end{equation}
where $\xi_0 = \min \{2\alpha_0, \beta_0\}$, for every $A$, $B\in \mathcal{L}(\RR^3,\RR^3)$.

In the sequel we will make use of the following norm in the linear
space of bounded isotropic tensors:
\[\|\CC\|_{\infty}=\max\left\{\|\lambda\|_{L^{\infty}(\om)},\|\mu\|_{L^{\infty}(\om)}\right\}.\]

Our boundary measurements are represented by the
Dirichlet-to-Neumann map. As a matter of fact, since we will
restrict our measurements to boundary data that have support on
some subset of the boundary, we will make use of a local
Dirichlet-to-Neumann map.

\begin{definition}[The Local Dirichlet-to-Neumann map]
\label{DNmap}
Let $\om$ be a bounded domain of Lipschitz class and let $\Sigma$ be an open portion of $\der\om$.
We denote by $H^{ \frac{1}{2} }_{co}(\Sigma)$ the function space
\[
H^{ \frac{1}{2} }_{co}(\Sigma):=\left\{\phi\in H^{ \frac{1}{2} }(\der\om)\,\ |\ \,\mbox{supp }\phi\subset\Sigma\right\}
\]
and by $H^{- \frac{1}{2} }_{co}(\Sigma)$ the topological dual of $H^{ \frac{1}{2} }_{co}(\Sigma)$. We denote by $<\cdot,\cdot>$ the dual pairing between $H^{ \frac{1}{2} }_{co}(\Sigma)$ and
$H^{- \frac{1}{2} }_{co}(\Sigma)$ based on the $L^2(\Sigma)$ scalar product, that is
$<f,g> = r_0^{-2}\int_{\partial\Omega}fg$, for every $f,g\in L^2(\partial\Omega)$.
 Then, given $\psi\in H^{ \frac{1}{2} }_{co}(\Sigma)$,  there exists a unique vector valued function $u\in H^1(\om)$ weak solution to the Dirichlet
 problem
\begin{equation}\label{1}
    \left\{\begin{array}{rcl}
             \mbox{div}(\CC  \nabhat u)& = & 0, \quad \mbox{ in }\om, \\
             u & = & \psi, \quad \mbox{ on }\der\om.
           \end{array}\right.
\end{equation}
We define the local Dirichlet-to-Neumann linear map
$\Lambda_{\CC}^{\Sigma}$ as follows:
\[
\Lambda_{\CC}^{\Sigma}:\psi\in  H^{ \frac{1}{2} }_{co}(\Sigma)\rightarrow
(\CC\nabhat u)n|_{\Sigma}\in H^{- \frac{1}{2} }_{co}(\Sigma),
\]
\end{definition}
where $n$ is the exterior unit vector to $\Omega$.

Note that for $\Sigma=\der\om$ we get the usual
Dirichlet-to-Neumann map. For this reason we will set
$\Lambda_{\CC}:=\Lambda_{\CC}^{\der\om}$.

The map $\Lambda_{\CC}^{\Sigma}$ can be identified with the bilinear form on $H^{ \frac{1}{2} }_{co}(\Sigma)\times H^{ \frac{1}{2} }_{co}(\Sigma)$ by
\begin{equation}\label{DNbilineare}
\widetilde{\Lambda}_{\CC}^{\Sigma}(\psi, \phi):=<\Lambda_{\CC}^{\Sigma}\psi, \phi>=r_0^{-2}\int_{\om}\CC \nabhat u:\nabhat v,
\end{equation}
for all $\psi,\phi \in H^{ \frac{1}{2} }_{co}(\Sigma)$ and where $u$ solves (\ref{1})
and $v$ is any $H^1(\om)$ function such that $v=\phi$ on $\der\om$.

We shall denote by $\|\cdot\|_{\star}$ the usual norm in the linear space $\mathcal{L}\left(H^{ \frac{1}{2} }_{co}(\Sigma),H^{- \frac{1}{2} }_{co}(\Sigma)\right)$.
Let us observe that, {}from our convention on the homogeneity of the norms,
we have in particular that

\[
\left\|\Lambda_{\CC}^{\Sigma}\right\|_{\star}=\sup
\left|r_0^{-2}\int_{\Omega}\CC\nabhat u:\nabhat
v\right|,
\]
where the sup is taken for $\phi$, $\psi\in H^{ \frac{1}{2} }_{co}(\Sigma)$,
$\|\phi\|_{ H^{ \frac{1}{2} }_{co}}(\Sigma) = \|\psi\|_{ H^{ \frac{1}{2}}_{co}}(\Sigma) =1$, being
$u$ the solution to (\ref{I1.1}) and $v\in H^1(\Omega)$ any extension of $\phi$.

\subsection{A priori assumptions and statement of the main result}\label{subsec2.3}

Our main assumptions are:

\bigskip

\noindent{\bf (A1)} $\om\subset\RR^3$ is an open bounded domain such that
\[
\Omega \ \hbox{is of class } C^{0, 1}, \ \hbox{with constants } \,
r_0, L,
\]
and we assume that
\[\ov{\om}=\cup_{j=1}^N \ov{D}_j,\]
where  $D_j$, $j=1,\ldots,N$, are connected and pairwise disjoint
domains of class $C^{0,1}$ with constants $r_{0}$, $L$, such that
there exists a constant $A>0$ such that
\[
|D_j|\leq Ar_0^3,\quad j=1,\cdots, N.
\]
We also assume that there exists one region, say $D_1$, such that
$\der D_1\cap\der\om$ contains the open  portion $\Sigma$ where
the measurements are taken. Moreover, for every
$j\in\{2,\ldots,N\}$ there exist $j_1,\ldots,j_M\in\{1,\ldots,N\}$
such that
\[
D_{j_1}=D_1,\quad D_{j_M}=D_j,
\]
and, for every $k=2,\ldots,M$, the set
\[
\der D_{j_{k-1}}\cap \der D_{j_k}
\]
contains a portion $\Sigma_k\subset\om$.

\noindent
Furthermore, for $k=1,\ldots,M$, we assume  there exists $P_k\in\Sigma_k$ and a rigid transformation of coordinates such that $P_k=0$ and for all $k=1,\cdots, M$
\label{walkway}\begin{eqnarray*}
\Sigma_k\cap Q_{r_0,L} &=& \{x\in Q_{r_0,L}\ |\  x_3=\varphi_k(x')\}, \\
D_{j_k}\cap Q_{r_0,L} &=& \{x\in  Q_{r_0,L}\ |\  x_3<\varphi_k(x')\}, \\
 D_{j_{k-1}}\cap  Q_{r_0,L} &=& \{x\in  Q_{r_0,L}\ |\
 x_3>\varphi_k(x')\},
\end{eqnarray*}
where $\Sigma_1\subset\Sigma$, with $\varphi_k\in C^{1,\alpha}(\mathbb{R}^2)$ such that
\[
\varphi_k(0)=|\nabla\varphi_k(0)|=0,\quad \|\varphi_k\|_{C^{1,\alpha}(\mathbb{R}^2)}\leq Lr_0,
\]
and where  we set  $D_{j_{0}}:=\mathbb{R}^3\backslash
\overline{\Omega}$.  Finally, let
\[
D_0:= \{x\in Q_{ \frac{2}{3}r_0,L} |
\varphi_1(x')<x_3<\frac{2}{3}r_0L\}
\]
and let $\Omega_0=\textit{Int}(\overline{\Omega\cup D_0})$.
Observe that $D_0$ and $\Omega_0$ are of class $C^{0,1}$, with
constants $r_1,L_1$ such that $L_1= \tan \left ( \frac{\pi}{4} +
\frac{1}{2} \arctan L \right )$ and $r_1 = \frac{2}{3} \frac{
\sqrt{  1+L^2  }}  { \sqrt{1+L_1^2}  }r_0$.

For simplicity we will call $D_{j_1},\ldots,D_{j_M}$ \textit{a
chain of domains connecting $D_1$ to $D_j$}. For any
$k\in\{1,\ldots,M\}$ we will denote by $n_k$ the exterior unit
vector to $\partial D_k$ in $P_k$.

\bigskip

\noindent{\bf (A2)} We assume that the tensor $\CC$ is piecewise
constant
\begin{equation}\label{isotrelisa}
\CC=\sum_{j=1}^N\CC_j \chi_{D_j}(x),
\end{equation}
where
\[\CC_j=\lambda_j I_3\otimes I_3 +2\mu_j{\II}_{Sym},\]
with constant Lam\'{e} coefficients $\lambda_j$ and $\mu_j$
satisfying \eqref{strong-convexity}.
In what follows we shall refer to the constants $L$, $\alpha$,
$A$, $N$, $\alpha_0$, $\beta_0$ as to the \emph{a priori data}.

In the sequel we will introduce a number of constants that we will
always denote by $C$. The values of these constants might differ
{}from one line to the other.

The main result of this paper is the following stability result.
\begin{theo}\label{teo2.1}
Let $\om$ and $\Sigma$ satisfy {\bf (A1)} and let the tensors $\CC$
and $\overline{\CC}$ satisfy {\bf (A2)}. Then there exists a positive
constant $C$ depending only on the a priori data such that
\begin{equation}\label{stabil}
\|\CC-\overline{\CC}\|_{\infty}\leq C r_0^{}\left\|\Lambda^{\Sigma}_{\CC}-\Lambda^{\Sigma}_{\overline{\CC}}\right\|_{\star}.
\end{equation}
\end{theo}

\subsection{Some basic properties of the Lam\'{e} system}\label{sec4}
\subsubsection{Alessandrini's identity}\label{subsec-aless}
Alessandrini's identity is a key relation connecting the
Dirichlet-to-Neumann maps and a volume integral. It was originally
derived in \cite{Al1} within the conductivity framework. Its
extension to our context is as follows. Given $u_1$ and $u_2$
solutions to
\[
 \mbox{div}(\CC^{k}  \nabhat u_k) =  0, \quad \mbox{ in }\om,
 \quad k=1,2,
\]
we have
\begin{equation}
    \label{AlessId}
    \int_{\om}(\CC^{1}-\CC^{2})\nabhat u_1:\nabhat
    u_2=r_0^2<(\Lambda_{\CC^{1}}-\Lambda_{\CC^{2}})u_2,u_1>,
\end{equation}
where $\Lambda_{\CC^{1}}$, $\Lambda_{\CC^{2}}$ denotes the
Dirichlet-to-Neumann map corresponding to $\CC^{1}$, $\CC^{2}$
respectively.

\subsubsection{Singular solutions}\label{subsec-sing}

In a suitable coordinate system, let us consider the set
\[
D=\{(x',x_3) \ |\ x_3<\varphi(x')\},
\]
where $\varphi\in C^{1,\alpha}(\mathbb{R}^2)$ is such that
\[
\varphi(0)=|\nabla\varphi(0)|=0,\quad \|\varphi\|_{C^{1,\alpha}(\mathbb{R}^2)}\leq Lr_0.
\]
Let
\[\CC_b=\CC+(\CC^D-\CC)\chi_{D},\]
where $\CC$ and $\CC^D$ are constant isotropic elasticity tensors
satisfying \eqref{strong-convexity}. Given $y\in\mathbb{R}^3$,
let us consider the \emph{normalized} fundamental solution
$\Gamma^D(\cdot,y)$ defined by

\begin{equation}
    \label{norm_fund_sol}
    \left\{\begin{array}{rl}
             \mbox{div}(\CC_b\nabhat \Gamma^D(\cdot,y)) & = -\delta_yI_3, \\
             \lim_{|x|\rightarrow\infty}\Gamma^D(x,y) & = 0.
           \end{array}\right.
\end{equation}

The following result, derived in \cite{AdCMR}, holds.

\begin{prop}\label{fundamental}
There exists a unique normalized fundamental solution
$\Gamma^D(\cdot, y) \in C^0(\mathbb{R}^3\setminus \{y\})$.
Moreover, for every $x \in \RR^3$, $x \neq y$, we have
\begin{equation}
  \label{bound1}
   \Gamma^D(x,y) = (\Gamma^D(y,x))^T,
\end{equation}
\begin{equation}
  \label{bound2}
   |\Gamma^D(x,y)| \leq C |x-y|^{-1},
\end{equation}
\begin{equation}
  \label{bound3}
   |\nabla_x \Gamma^D(x,y)| \leq C |x-y|^{-2},
\end{equation}
where the constant $C>0$ only depends on $L$, $\alpha$, $\alpha_0$
and $\beta_0$.
\end{prop}
In particular, for $\varphi=0$, we have $D=\R_-^3$ and we
will denote the fundamental solution by $\Gamma$. An explicit
expression of $\Gamma$ has been obtained by Rongved in \cite{R}.

A crucial result in our analysis is the following asymptotic
estimate of $\Gamma^D$ in terms of $\Gamma$ derived in
\cite[Theorem 8.1]{AdCMR}.
\begin{prop}\label{asymptotic}
Let $y=(0,0,h)$, where $0<h< \frac{L r_0}{8 \sqrt{1+L^2} }$. Then
\begin{equation}\label{asymp1}
|(\Gamma^D-\Gamma)(x,y)|\leq
\frac{C}{r_0}\left(\frac{|x-y|}{r_0}\right)^{-1+\alpha},\quad
\forall x\in Q_{\frac{r_0}{8 \sqrt{1+L^2} },L}\cap D,
\end{equation}
\begin{equation}\label{asymp2}
|(\nabla_x\Gamma^D-\nabla_x\Gamma)(x,y)|\leq
\frac{C}{r^2_0}\left(\frac{|x-y|}{r_0}\right)^{-2+\frac{\alpha^2}{3\alpha+2}},\quad
\forall x\in Q^-_{\frac{r_0}{12 \sqrt{1+L^2} }, L}\cap D,
\end{equation}
where $C$ only depends on $L,\alpha, \alpha_0$ and $\beta_0$.
\end{prop}
Let $\CC$ be an isotropic elasticity tensor satisfying {\bf
(A2)}. We still denote by $\CC$ its  extension to $\Omega_0$ such
that $\CC|_{D_0}=\CC_0$ is the isotropic tensor with Lam\'{e}
parameters $\lambda_0=0$ and $\mu_0=1$. This extended tensor is
still an isotropic elasticity tensor of the form
\begin{equation}\label{isotr1}
\CC=\sum_{j=0}^N\CC_j \chi_{D_j}(x),
\end{equation}
where each $\CC_j$, $j=0,\ldots,N$, has Lam\'{e} parameters
satisfying \eqref{strong-convexity}.

For all possible interfaces $\Sigma_k$ introduced in \textbf{(A1)}
let
\[
\Sigma_k'=\Sigma_k\cap Q_{\frac{r_0}{3}, L},
\]
and denote by $\mathcal{F}:=\cup_{k}\Sigma_k'$. Let
\[y\in \cup_{j=0}^N D_j\cup \mathcal{F}, \quad r^*=r^*(y)=\min \left \{\frac{r_0}{3}, \textit{dist}(y, \cup_{j=0}^N \der D_j\backslash\mathcal{F}) \right
\}.
\]
Let us consider the sphere  $B_{r^*}(y)$. Then, either $B_{r^*}(y)\cap
\mathcal{F}=\emptyset$, so that  $B_{r^*}(y)\subset D_j$  for some
$j\in \{1,\cdots,N\}$ and  we define $\CC_y=\CC_j$, or
$B_{r^*}(y)\cap  \mathcal{F}\neq\emptyset$ and under our
regularity assumptions there exist exactly two domains, say,
$D_{j-1}$ and $D_j$, intersecting $B_{r^*}(y)$ and, in this case,
we define
$\CC_y=\CC_{j-1}+(\CC_{j}-\CC_{j-1})\chi_{\{x_3<\varphi_j(x')\}}$.
In the latter expression, $\varphi_j$ is the function whose graph
contains $\Sigma_j$, according to \textbf{(A1)}.

Let $\Gamma'(\cdot,y)$ denote the normalized fundamental solution
to

\begin{equation}\label{fundamental_solution}
\mbox{div}( \CC_y\nabhat \Gamma'(\cdot, y))=-\delta_y I_3, \quad
\text{ in }\RR^3.
\end{equation}

\begin{prop}\label{Green} Let $\om_0$ and  $\CC$ satisfy
\textbf{(A1)} and \textbf{(A2)}. Then, for any $y\in \cup_{j=0}^N
D_j\cup\mathcal{F}$, there exists a unique matrix-valued
function $G(\cdot,y)\in C^0(\Omega_0 \setminus \{y\},
\mathcal{L}(\RR^3,\RR^3))$ such that
\begin{equation}
    \label{green-2}
    \int_{\om_0}\CC\nabhat G(\cdot,y):\nabhat \phi=\phi(y), \quad \forall\phi\in C^\infty_0(\om_0, \mathcal{L}(\RR^3,\RR^3)),
\end{equation}
\begin{equation}
    \label{green-1}
    G(\cdot,y)=0, \quad \text{ on }\der\om_0,
\end{equation}
and
\begin{equation}\label{green1}
    \|G(\cdot,y)\|_{H^1(\om_0\setminus B_r(y))}\leq \frac{C}{\sqrt{rr_0}}, \quad\forall r\leq
    r^*,
\end{equation}
where $C>0$ depends only on $\alpha_0$, $\beta_0$, $A$, $N$, $\alpha$ and $L$.\\
Furthermore, for every $c_1>1$, if  $\textit{dist}(y, \cup_{j=0}^N
\der D_j\backslash\mathcal{F})\}\geq \frac{r_0}{c_1}$, then
\begin{equation}\label{energy}
\|G(\cdot,y)-\Gamma'(\cdot, y)\|_{H^1(\om_0)}\leq \frac{C}{r_0},
\end{equation}
where $C>0$ depends only on $\alpha_0$, $\beta_0$, $A$, $N$,
$\alpha$, $L$ and $c_1$. Finally
\begin{equation}\label{symgreen}
G(x,y)=G(y,x)^T, \quad \mbox{for every }x,y \in \left (
\cup_{j=0}^N D_j \right ) \cup \mathcal{F}, \  \ x \neq y.
\end{equation}
\end{prop}
The proof follows the lines of the proof of Proposition 3.1 in
\cite{BFV}.

\subsubsection{Three spheres inequality}\label{subsec-3spheres}

A mathematical tool which plays an important role in the proof of
Theorem \ref{teo2.1} is the following three spheres inequality for
solutions to the Lam\'{e} system.

\begin{prop}\label{3spheres}
Let $u \in H^1(B_R)$ be a solution to the Lam\'{e} system
\begin{equation}
    \label{3s-1}
    \mbox{div}( \CC \widehat{\nabla} u) = 0, \quad \hbox{in } B_R,
\end{equation}
where $\CC$ is a constant isotropic elasticity tensor satisfying
\eqref{strong-convexity}. For every $r_1$, $r_2$, $r_3$, with $0 <
r_1 < r_2 < r_3 \leq R$, we have
\begin{equation}
    \label{3s-2}
    \|u\|_{L^\infty(B_{r_2})} \leq C
    \|u\|_{L^\infty(B_{r_1})}^\delta
    \|u\|_{L^\infty(B_{r_3})}^{1-\delta},
\end{equation}
where $C>0$ and $\delta \in (0,1)$ only depend on $\alpha_0$,
$\beta_0$, $ \frac{r_2}{r_3}$, $ \frac{r_1}{r_3}$.
\end{prop}
For a proof, see \cite{Al-M}.

\section{Proof of the main result}\label{sec5}

Let $j\in\{1,\ldots,N\}$ be such that
\[
\|\CC-\overline{\CC}\|_{L^{\infty}(D_j)}=\|\CC-\overline{\CC}\|_{L^{\infty}(\Omega_0)}
\]
and let $D_{j_1},\ldots,D_{j_M}$ be a chain of domains, defined
according to \textbf{(A1)}, connecting $D_1$ to $D_j$. For the
sake of brevity, set $D_k=D_{j_k}$, $k=1,\ldots, M$. Let
\begin{equation}
    \label{path-and-complem}
    {\cal W}_k=\textit{Int}(\cup_{j=0}^k \overline{D}_j), \quad {\cal
    U}_k=\Omega_0\backslash \overline{{\cal W}_k}, \quad k=0,\dots, M-1.
\end{equation}
Note that here the tensors $\CC$ and $\overline{\CC}$ are extended as in
(\ref{isotr1}) in all of $\Omega_0$. Finally, for
$y,z\in(\cup_{j=0}^k D_j)\cup (\cup_{j=1}^k \Sigma'_j)$, let us
define the matrix-valued function
\[
{\cal S}_k(y,z):=\int_{{\cal U}_k}(\CC-\overline{\CC})(x)\nabhat
G(x,y):\nabhat  \overline{G}(x,z)\, dx ,
\]
whose entries are given by
\[
{\cal S}^{(p,q)}_k(y,z):=\int_{{\cal U}_k}(\CC-\overline{\CC})(x)\nabhat
G(x,y)e_p:\nabhat \overline{G}(x,z)e_q\,dx,\quad p,q=1,2,3,
\]
and where $G(\cdot,y)$ and $\overline{G}(\cdot,z) $ denotes
respectively the
singular solution of Proposition
 \ref{Green} corresponding to the tensors $\CC$ and $\overline{\CC}$,
 respectively. Let us denote ${\cal S}_k^{(\cdot,q)}= \sum_{i=1}^3{\cal
 S}_k^{(i,q)}e_i$, ${\cal S}_k^{(p, \cdot)}=  \sum_{i=1}^3{\cal
 S}_k^{(p,i)}e_i$, where $e_i$, $i=1,2,3$, are the fundamental
 unit vectors of $\R^3$.

Proceeding similarly to in \cite[Proposition 4.4]{BFV}, one can
see that the functions ${\cal S}_k^{(\cdot,q)}$, ${\cal S}_k^{(p, \cdot)}$ are solutions
to the Lam\'e system with elasticity tensor $\CC$ defined in \eqref{isotr1} in the weak sense
clarified below.

\begin{prop}\label{sol}
Let $\mathcal{R}_k:=(\cup_{j=0}^k \der D_j)\backslash
(\cup_{j=1}^k \Sigma'_j)$ and let $y,z\in(\cup_{j=0}^k D_j)\cup
(\cup_{j=1}^k \Sigma'_j)$.  Then, ${\cal S}_k^{(\cdot,q)}(\cdot,
z)$ and ${\cal
S}_k^{(p,\cdot)}(y, \cdot)$ belong to
$H^1_{loc}(\mathcal{W}_k\backslash\mathcal{R}_k, \RR^3)$ and
\begin{equation}\label{soly}
    \int_{\mathcal{W}_k}\CC  \nabhat_y {\cal S}_k^{(\cdot,q)}(y,z):\nabla\phi(y)\, dy=0,\quad\forall \phi\in C^{\infty}_0(\mathcal{W}_k\backslash\mathcal{R}_k,
    \RR^3), \quad \forall q\in\{1,2,3\},
\end{equation}
\begin{equation}\label{solz}
\int_{\mathcal{W}_k}\CC  \nabhat_z {\cal
S}_k^{(p,\cdot)}(y,z):\nabla\phi(z)\, dz=0,\quad\forall \phi\in
C^{\infty}_0(\mathcal{W}_k\backslash\mathcal{R}_k, \RR^3), \quad
\forall p\in\{1,2,3\}.
\end{equation}
\end{prop}
\begin{proof}[Proof of Theorem \ref{teo2.1}]
Let
 \[
 E:=\|\CC-\overline{\CC}\|_{L^{\infty}(\Omega_0)}
 \]
 \[
 \epsilon:=r_0\left\|\Lambda_{\CC}^\Sigma-\Lambda_{\overline{\CC}}^\Sigma\right\|_\star
 \]
and let
\[
K_0:=\left \{ x\in D_0 | \textit{dist}(x,\der\Omega)\geq
\frac{r_0}{24} \right \}.
\]
By H\"older inequality and (\ref{green1}) we have
\begin{equation}\label{aprioriest}
    |\mathcal{S}_k^{(p,q)}(y,z)|\leq
    \frac{CE}{\sqrt{d(y)d(z)}},\quad\forall \ p,q=1,2,3,
\end{equation}
where $d(y)=d(y,\mathcal{U}_k)$, $d(z)=d(z,\mathcal{U}_k)$ and
$C>0$ only depends on $\alpha_0$,$\beta_0$, $A$, $N$, $\alpha$ and
$L$.

By Alessandrini's identity \eqref{AlessId} applied to
$u_1(\cdot)=G(\cdot,y)l$ and $u_2(\cdot)=\overline{G}(\cdot,z)m$, for
$y,z\in K_0$ and for $l,m\in\mathbb{R}^3$ with $|l|=|m|=1$, we get
\begin{equation}
    \label{small}
    \left|\int_{\om_0}(\CC-\overline{\CC})(x)\nabhat G(x,y)\,l:\nabhat\overline{G}(x,z)\,m\, dx\right| \leq \frac{C}{r_0}\varepsilon,
\end{equation}
where $C>0$ depends on the a priori data only.\\
We now proceed iteratively with respect to the index $k$. \\
{\bf First step: $k=0$.}\\
For $y,z\in K_0$, let us consider
\[
{\cal S}_0^{(p,q)}(y,z):=\int_{{\cal U}_0}(\CC-\overline{\CC})(x)\nabhat
G(x,y)e_p:\nabhat \overline{G}(x,z)e_q\, dx, \quad p,q \in \{1,2,3\}.
\]
{}From (\ref{small}) we get
\[
|{\cal S}_0^{(p,q)}(y,z)|\leq \frac{C}{r_0}\epsilon,\quad \forall
y,z\in K_0,
\]
where $C>0$ depends only on the a priori data. Let us fix $z\in K_0$ and $q\in\{1,2,3\}$.
Recalling that, for fixed $q\in\{1,2,3\}$, ${\cal S}^{(\cdot,q)}_0(\cdot ,z)$ is solution to
(\ref{soly}), we shall propagate the smallness with respect to the
first variable {}from the point $Q_1=P_1 + \frac{r_0}{6}L n_1$ to
$y_r=P_1+rn_1$, for every $r \in \left(0,
\frac{r_0}{24\sqrt{1+L^2}}\right]$, by iterating the three spheres
inequality \eqref{3s-2} over a chain of balls of decreasing
radius and contained in a suitable cone with vertex at $P_1$ and
axis in the direction $n_1$, obtaining
\begin{equation}
    \label{holderest}
    |{\cal S}^{(\cdot,q)}_0(y_r,z)|\leq \frac{CE}{\sqrt{rr_0}}\left(\frac{\epsilon}{E}\right)^{\delta^{\eta}},
\end{equation}
where
\begin{equation}
   \label{eta_def}
y_r=P_1+rn_1, \ \eta=b\left|\ln \left(a\frac{r}{r_0}\right)\right|+1, \quad  r
\in \left(0, \frac{r_0}{24\sqrt{1+L^2}}\right],
\end{equation}
$\delta\in (0,1)$ only depends on $\alpha_0$, $\beta_0$, and $a>0,\ b>0$ only
depend on $L$.

Now, for fixed $p\in\{1,2,3\}$, let us consider a solution ${\cal
S}^{(p,\cdot)}_0(y_r,\cdot)$ to (\ref{solz}). Then, a similar
procedure leads to
\begin{equation}\label{holderest2}
    |{\cal S}^{(p, \cdot)}_0(y_r,z_{ \overline{r}})|\leq
    \frac{CE}{\sqrt{\overline{r}r}}\left(\frac{\epsilon}{E}\right)^{\delta^{\eta+\overline{\eta}}},
\end{equation}
where
\[
z_{\overline{r}}=P_1+\overline{r}n_1,
\quad\overline\eta=b\left|\ln a\frac{\overline{r}}{r_0}\right|+1, \quad \overline{r} \in \left(0, \frac{r_0}{24\sqrt{1+L^2}}\right].
\]
Hence, for every $r, \overline{r} \in \left(0, \frac{r_0}{24\sqrt{1+L^2}}\right]$
\begin{equation}\label{Small1}
|{\cal S}^{(p, q)}_0(y_r,z_{\overline{r}})|\leq
\frac{CE}{\sqrt{\overline{r}r}}\left(\frac{\epsilon}{E}\right)^{\delta^{\eta+\overline{\eta}}},
\quad \forall \  p,q=1,2,3.
\end{equation}

Let $\overline{r}=cr$, with $c\in \left[\frac{2}{3}, \frac{4}{5}\right]$, $p=q=3$
and let $r_1=\frac{r_0}{12\sqrt{1+L^2}}$. Let us write

\begin{equation}\label{split}
{\cal S}_0(y_r,z_{\overline{r}})\,e_3\cdot e_3=I_1+I_2,
\end{equation}
where
\begin{equation}\label{I1}
    I_1= \int_{B_{r_1}\cap D_{1}}\!\!\!\!\!\!(\CC-\overline{\CC})(x)\nabhat G(x,y_r) \,e_3:\nabhat\overline{G}(x,z_{\overline{r}})\,e_3\,dx
\end{equation}
\begin{equation}\label{I2}
I_2= \int_{\Omega\backslash (B_{r_1}\cap
D_{1})}\!\!\!\!\!\!(\CC-\overline{\CC})(x)\nabhat G(x,y_r)
\,e_3:\nabhat\overline{G}(x,z_{\overline{r}})\,e_3\,dx \ .
\end{equation}
Here and in the sequel, $B_{r_1}$ denotes $B_{r_1}(P_1)$. Then,
{}from (\ref{green1}) we have
\begin{equation}\label{Small4}
|I_2|\leq \frac{CE}{r_0}.
\end{equation}
{}From (\ref{Small1}), (\ref{split}) and (\ref{Small4}) we derive
\begin{equation}\label{Bound1}
    |I_1|\leq CE\left( \frac{1}{r}\left(\frac{\epsilon}{E}\right)^{\delta^{\eta+\overline{\eta}}}+ \frac{1}{r_0}\right).
\end{equation}
We rewrite $I_1$ as follows
\begin{equation}\label{decomp}
I_1=\int_{B_{r_1}\cap D_{1}}\!\!\!\!\!\!(\CC_1-\overline{\CC}_1)\nabhat
\Gamma'(x,y_r) \,e_3:\nabhat \overline{\Gamma}'(x,z_{\overline{r}})\,e_3\,dx+A_1+A_2+A_3,
\end{equation}
where $\Gamma'$, $\overline{\Gamma'}$ is the normalized
fundamental solution to (\ref{norm_fund_sol}) corresponding to the pair of elasticity tensors
$(\CC=\CC_0,\CC^D=\CC_1)$, $(\CC=\overline{\CC}_0,\CC^D=\overline{\CC}_1)$ respectively, $D=D_1$ and
\[
A_1=\int_{B_{r_1}\cap D_{1}}\!\!\!\!\!\!(\CC_1-\overline{\CC}_1)(\nabhat
G-\nabhat\Gamma')(x,y_r) \,e_3:(\nabhat\overline{G}-\nabhat\overline{\Gamma}')(x,z_{\overline{r}})\,e_3\,dx \ ,
\]
\[
A_2=\int_{B_{r_1}\cap D_{1}}\!\!\!\!\!\!(\CC_1-\overline{\CC}_1)(\nabhat
G-\nabhat\Gamma')(x,y_r) \,e_3:\nabhat\overline{\Gamma}'(x,z_{\overline{r}
})\,e_3\,dx \ ,
\]
\[
A_3=\int_{B_{r_1}\cap
D_{1}}\!\!\!\!\!\!(\CC_1-\overline{\CC}_1)\nabhat\Gamma'(x,y_r)
\,e_3:(\nabhat\overline{G}-\nabhat\overline{\Gamma}')(x,z_{\overline{r}})\,e_3\,dx \
.
\]
By \eqref{energy} and (\ref{bound3}) we obtain
\begin{equation}\label{Bound2}
|A_1|\leq \frac{CE}{r_0} ,
\end{equation}
\begin{equation}\label{Bound3}
|A_2|, |A_3|\leq \frac{CE}{\sqrt{rr_0}} ,
\end{equation}
where $C>0$ only depends on the a priori data. {}From (\ref{decomp})-(\ref{Bound3}) we get
\begin{equation}\label{Bound4}
|I_1|\geq\left| \int_{B_{r_1}\cap
D_{1}}\!\!\!\!\!\!(\CC_1-\overline{\CC}_1)\nabhat \Gamma'(x,y_r)
\,e_3:\nabhat\overline{\Gamma}'(x,z_{\overline{r}
})\,e_3\,dx\right|-\frac{CE}{\sqrt{rr_0}} ,
\end{equation}
where $C>0$ only depends on the a priori data. {}From (\ref{Bound1}) and (\ref{Bound4}) we
obtain
\begin{equation}
    \label{Bound3bis}
    \left| \int_{B_{r_1}\cap D_{1}}\!\!\!\!\!\!(\CC_1-\overline{\CC}_1)\nabhat \Gamma'(x,y_r) \,e_3:\nabhat\overline{\Gamma}'(x,z_{\overline{r}})\,e_3\,dx\right|\leq
    \frac{CE}{r}\left( \left(\frac{\epsilon}{E}\right)^{\delta^{\eta+\overline{\eta}}}+ \left ( \frac{r}{r_0}  \right )^{\frac{1}{2}}
    \right)\ ,
\end{equation}
where $C>0$ only depends on the a priori data.

Let us denote by $\Gamma$ and $\overline{\Gamma}$ the Rongved fundamental
solutions corresponding to the tensors
$\CC_0\chi_{\mathbb{R}^3_+}+\CC_1\chi_{\mathbb{R}^3_-}$ and
$\overline{\CC}_0\chi_{\mathbb{R}^3_+}+\overline{\CC}_1\chi_{\mathbb{R}^3_-}$,
respectively. Let
\begin{equation}
    \label{decomp2}
    \int_{B_{r_1}\cap D_{1}}\!\!\!\!\!\!(\CC_1-\overline{\CC}_1)\nabhat \Gamma'(x,y_r) \,e_3:\nabhat\overline{\Gamma}'(x,z_{\overline{r}
    })\,e_3\,dx=:B_1+B_2+B_3\ ,
\end{equation}
where
\[
B_1=\int_{B_{r_1}\cap D_{1}}\!\!\!\!\!\!(\CC_1-\overline{\CC}_1)\nabhat
\Gamma(x,y_r) \,e_3:\nabhat\overline{\Gamma}(x,z_{\overline{r}})\,e_3\,dx \ ,
\]
\[
B_2=\int_{B_{r_1}\cap D_{1}}\!\!\!\!\!\!(\CC_1-\overline{\CC}_1)\nabhat
(\Gamma'-\Gamma)(x,y_r) \,e_3:\nabhat\overline{\Gamma}(x,z_{\overline{r}
})\,e_3\,dx \ ,
\]
\[
B_3=\int_{B_{r_1}\cap
D_{1}}\!\!\!\!\!\!(\CC_1-\overline{\CC}_1)\nabhat\Gamma'(x,y_{
r})\,e_3:(\nabhat\overline{\Gamma}'-\nabhat\overline{\Gamma})(x,z_{\overline{r}})
\,e_3\,dx \ .
\]
To estimate $B_2$ and $B_3$, we observe that
\[
B_{r_1}\cap D_{1}\subset \left(D_1\cap
Q^-_{\frac{r_0}{12\sqrt{1+L^2}},L}\right)\cup\left\{(x',x_3): 0\leq
x_3\leq \frac{L}{r_0^{\alpha}}|x'|^{1+\alpha}\right\} \ .
\]
In $D_1 \cap Q^-_{ \frac{r_0}{ 12 \sqrt{1+L^2}  }}$ we can apply
the asymptotic estimate \eqref{asymp2} so that, recalling also
\eqref{bound3}, we have
\[
|B_2| \leq \frac{CE}{r_0^\gamma} \int_{\RR_{-}^3} |x-y_r|^{\gamma
-2} |x -z_{\overline{r}}|^{-2} dx +  CE \int_{0\leq x_3 \leq
\frac{L}{r_0^\alpha}|x'|^{1+\alpha}}
|x-y_r|^{-2}|x-z_{\overline{r}}|^{-2} dx \ ,
\]
where $\gamma = \frac{\alpha^2}{3\alpha+2} < \frac{1}{2}$ and
$C>0$ depend only on $\alpha_0$, $\beta_0$, $L$ and $\alpha$. The
first integral can be easily estimated by passing to cylindrical
coordinates and by applying H\"{o}lder inequality, obtaining
\[
    \frac{CE}{r_0^\gamma} \int_{\RR_{-}^3} |x-y_r|^{\gamma -2} |x
    -z_{\overline{r}}|^{-2} dx  \leq \frac{CE}{r_0} \left (
    \frac{r}{r_0} \right )^{\gamma -1} \ .
\]
The estimate of the second integral is not straightforward. First,
by performing the change of variables $y= \frac{x}{r}$, we have
\begin{multline*}
    CE \int_{0\leq x_3 \leq \frac{L}{r_0^\alpha}|x'|^{1+\alpha}}
    |x-y_r|^{-2}|x-z_{\overline{r}}|^{-2} dx \leq
    \\
    \leq \frac{CE}{r} \int_{\RR^2} \left ( \int_0^{
    \frac{L}{r_0^\alpha}r^\alpha |y'|^{1+\alpha}} \frac{1}{ (|y'|^2
    +(y_3-1)^2)(|y'|^2 +(y_3 -c)^2)} dy_3 \right ) dy_1 dy_2 .
\end{multline*}
By splitting $\RR^2$ as the union of $A= \left\{ y' \in \RR^2 | \ |y'|
\geq \left ( \frac{2L}{r_0^\alpha} \right )^{-
\frac{1}{1+\alpha}}r^{ - \frac{\alpha}{1+\alpha}} \right\}$ and $B= \RR^2\setminus A $, we have
\begin{equation*}
    CE \int_{0\leq x_3 \leq \frac{L}{r_0^\alpha}|x'|^{1+\alpha}}
    |x-y_r|^{-2}|x-z_{\overline{r}}|^{-2} dx \leq
    \frac{CE}{r_0} \left(\left ( \frac{r}{r_0} \right )^{ \frac{\alpha -1 }{\alpha
    +1}}+ \left ( \frac{r}{r_0} \right )^{ \alpha -1}\right),
\end{equation*}
with $C>0$ only depending on $\alpha_0$, $\beta_0$, $L$ and
$\alpha$, where we have used the fact that $|y_3| < \frac{1}{2}$
in $B$, so that $|y_3 - 1| > \frac{1}{2}$ and $|y_3 -c| \geq c-
\frac{1}{2} \geq \frac{1}{6}$.

The term $B_3$ in \eqref{decomp2} can be estimated similarly,
obtaining
\begin{equation}\label{Bound5}
    |B_2|,|B_3|\leq
    \frac{CE}{r_0}\left(\left(\frac{r}{r_0}\right)^{\gamma-1}+\left(\frac{r}{r_0}\right)^{\alpha-1}\right),
\end{equation}
where $C>0$ only depends on $\alpha_0$, $\beta_0$, $L$, $\alpha$,
and $\gamma= \frac{\alpha^2}{3\alpha + 2 } <\frac{1}{2}$.

We split $B_1$ as follows
\[
B_1=C_1+C_2+C_3 ,
\]
where
\[
C_1=\int_{B^-_{r_1}}\!\!\!\!\!\!(\CC_1-\overline{\CC}_1)\nabhat
\Gamma(x,y_r) \,e_3:\nabhat \overline{\Gamma}(x,z_{\overline{r}})\,e_3\,dx \ ,
\]
\[
C_2=\int_{B^+_{r_1}\cap D_1}\!\!\!\!\!\!(\CC_1-\overline{\CC}_1)\nabhat
\Gamma(x,y_r) \,e_3:\nabhat\overline{\Gamma}(x,z_{\overline{r}})\,e_3\,dx \ ,
\]
\[
C_3=\int_{B^-_{r_1}\backslash
D_1}\!\!\!\!\!\!(\CC_1-\overline{\CC}_1)\nabhat \Gamma(x,y_r)
\,e_3:\nabhat\overline{\Gamma}(x,z_{\overline{r}})\,e_3\,dx \ .
\]
Since
\[
B_{r_1}^+\cap D_{1}\subset \left\{(x',x_3)| 0\leq x_3\leq
\frac{L}{r_0^{\alpha}}|x'|^{1+\alpha}\right\}
\]
and
\[
B^-_{r_1}\backslash D_1\subset \left\{(x',x_3)|
-\frac{L}{r_0^{\alpha}}|x'|^{1+\alpha}\leq x_3\leq 0\right\},
\]
we can estimate the terms $C_2$ and $C_3$ similarly to the second
addend of $B_2$, getting
\begin{equation}
    \label{Bound6}
    |C_2|,|C_3|\leq
    \frac{CE}{r_0}\left(\frac{r}{r_0}\right)^{\alpha-1},
\end{equation}
where $C>0$ only depends on $\alpha_0$, $\beta_0$, $L$, $\alpha$.

Finally, to estimate $C_1$, we use the following property of the
Rongved fundamental solution
\[
\Gamma(\xi,y_0)=h\Gamma(h\xi,hy_0),
\quad\overline{\Gamma}(\xi,y_0)=h\overline{\Gamma}(h\xi,hy_0), \quad \forall
\xi\neq y_0, \ \forall h>0.
\]
Then
\begin{equation}\label{changevariable}
C_1=\frac{1}{r}\int_{B^-_{\frac{r_1}{r}}}\!\!\!\!\!\!(\CC_1-\overline{\CC}_1)(x)\nabhat
\Gamma(x,e_3) \,e_3:\nabhat\overline{\Gamma}(x,ce_3)\,e_3\,dx
\end{equation}
and by (\ref{Bound3bis})--(\ref{Bound6}) we obtain
\begin{equation}\label{Bound7}
\left|\int_{B^-_{\frac{r_1}{r}}}\!\!\!\!\!\!(\CC_1-\overline{\CC}_1)(x)\nabhat \Gamma(x,e_3) \,e_3:\nabhat\overline{\Gamma}(x,ce_3)\,e_3\,dx\right|
\leq
CE \left ( \left(\frac{r}{r_0}\right)^{\gamma}+\left(\frac{\epsilon}{E}\right)^{\delta^{\eta+\overline{\eta}}}\right)
\ ,
\end{equation}
where $C>0$ only depends on the a priori data.
{}From \eqref{bound3} and since $c\in \left[\frac{2}{3}, \frac{4}{5}\right]$
and $r_1=\frac{r_0}{12\sqrt{1+L^2}}$, we derive
\begin{equation}\label{Bound8}
\left|\int_{\mathbb{R}_-^3\backslash {B^-_{\frac{r_1}{r}}}}\!\!\!\!\!\!(\CC_1-\overline{\CC}_1)(x)\nabhat \Gamma(x,e_3) \,e_3:\nabhat\overline{\Gamma}(x,ce_3)\,e_3\,dx\right|\\
\leq CE\frac{r}{r_0} \ ,
\end{equation}
where $C>0$ only depends on $\alpha_0$, $\beta_0$, $L$, $\alpha$.
{}From (\ref{Bound7}), (\ref{Bound8}), since $\gamma<\alpha$,
$\gamma<\frac{1}{2}$ and $c\geq \frac{2}{3}$, by (\ref{holderest})
we have
\begin{equation}\label{Bound9}
\left|\int_{\mathbb{R}_-^3}\!\!\!\!\!\!(\CC_1-\overline{\CC}_1)(x)\nabhat \Gamma(x,e_3) \,e_3:\nabhat\overline{\Gamma}(x,ce_3)\,e_3\,dx\right|\\
\leq CEf\left(\frac{r}{r_0}\right), \quad \forall r\leq
\frac{r_0}{24\sqrt{1+L^2}},
\end{equation}
where
\[
f(\rho)=\rho^{\gamma}+\left(\frac{\epsilon}{E}\right)^{\delta^{2+A|\ln
B\rho|}} \ ,
\]
where $0 < \rho\leq \frac{1}{24\sqrt{1+L^2}}$, and $A$, $B>0$ only depend on $L$.

By an appropriate choice of $\rho=\rho(\epsilon)$, we get
\begin{equation}\label{Bound9bis}
    \left|\int_{\mathbb{R}_-^3}\!\!\!\!\!\!(\CC_1-\overline{\CC}_1)(x)\nabhat \Gamma(x,e_3) \,e_3:\nabhat\overline{\Gamma}(x,ce_3)\,e_3\,dx\right|\\
    \leq
    CE\left|\ln\frac{\epsilon}{E}\right|^{\frac{-\gamma}{2A|\ln\delta|}}\
    ,
\end{equation}
where $C>0$ only depends on $\alpha_0$, $\beta_0$, $L$, $\alpha$.
Applying Proposition 3.2 of \cite{BFV} we have
\begin{equation*}
    \int_{\mathbb{R}_-^3}\!\!\!\!\!\!(\CC_1-\overline{\CC}_1)(x)\nabhat \Gamma(x,e_3) \,e_3:\nabhat\overline{\Gamma}(x,ce_3)\,e_3\,dx
    =
    (\Gamma(e_3,ce_3)-\overline{\Gamma}(e_3,ce_3))e_3 \cdot e_3
\end{equation*}
and then
\begin{equation}
    \label{Bound10}
    \left|\Gamma(e_3,ce_3)-\overline{\Gamma}(e_3,ce_3)\right|\leq
    CE\left|\ln\frac{\epsilon}{E}\right|^{\frac{-\gamma}{2A|\ln\delta|}}\
    ,
\end{equation}
where $C>0$ only depends on the a priori data.
Now, by using the explicit form of the Rongved fundamental
solution and proceeding as in \cite{BFV} (Section 4.2), it can be
shown that \eqref{Bound10} implies
\begin{equation}
    \label{Bound11}
    \|\CC_1-\overline{\CC}_1\|_\infty \leq CE\omega_1\left(\frac{\epsilon}{E}\right),
\end{equation}
where
\[
\omega_1(t)=|\ln t|^{\frac{-\gamma}{2A|\ln\delta|}}
\]
and $C>0$ only depends on the a priori data.

If  $\|\CC-\overline{\CC}\|_{\infty}=\|\CC_1-\overline{\CC}_1\|_\infty$, then we
get
\[
\|\CC-\overline{\CC}\|_{\infty}=E\leq
\frac{\epsilon}{\omega_1^{-1}(\frac{1}{C})} \ ,
\]
and the claim follows. Otherwise, we proceed with the next step. \\
{\bf Second step: $k=1$}

In this case, let us consider the function
\begin{multline*}
{\cal S}_1^{(p,q)}(y,z)=\int_{\Omega}(\CC-\overline{\CC})(x)\nabhat G(x,y)e_p:\nabhat \overline{G}(x,z)e_q\, dx-\\
-\int_{D_1}(\CC-\overline{\CC})(x)\nabhat G(x,y)e_p:\nabhat \overline{G}(x,z)e_q\, dx,\quad p,q=1,2 ,3.
\end{multline*}
{}From (\ref{green1}), (\ref{small}) and (\ref{Bound11}) we get
\begin{equation}\label{Bound12}
|{\cal S}_1^{(p,q)}(y,z)|\leq
\frac{CE}{r_0}\omega_1\left(\frac{\epsilon}{E}\right),\quad \forall\ y,z\in K_0,
\end{equation}
where $C>0$ only depends on the a priori data.
Proceeding similarly to what previously described for ${\cal S}_0^{(\cdot,q)}(\cdot,z)$, and by regularity estimates  for elliptic systems, we derive that for every $P\in\Sigma'_1$ and every $r$, $0<r\leq \frac{r_0}{24\sqrt{1+L^2}}$,
\begin{equation}\label{holderest3}
|{\cal S}_1^{(p,q)}(y'_r,z)|+r_0 |\nabla{\cal S}_1^{(p,q)}(y'_r,z)|\leq \frac{CE}{r_0}\left(\omega_1\left(\frac{\epsilon}{E}\right)\right)^{\delta^{\eta}},\quad\forall z\in K_0,
\end{equation}
where $y'_r=P+rn_1$ and $\eta$ is the same constant defined in (\ref{eta_def}).
Let us recall the following regularity estimates due to Li and Nirenberg \cite{LN}
\begin{equation}\label{Bound13}
\|\nabla{\cal S}_1^{(p,q)}(\cdot,z)\|_{L^{\infty}( \overline{D}_0\cap Q_{\frac{r_0}{3},L} )}+r_0^{\beta}\left|\nabla{\cal S}_1^{(p,q)}(\cdot,z)\right|_{\beta, \overline{\Gamma}_0\cap Q_{\frac{r_0}{3},L} }\leq \frac{CE}{r^2_0},
\end{equation}
where $\beta=\frac{\alpha}{2(1+\alpha)}$ and $C>0$ only depends on the a priori data.

{}From (\ref{holderest3}) and (\ref{Bound13}) we have that, for every $P\in \Sigma'_1$, $z\in K_0$ and $r\in \left(0, \frac{r_0}{24\sqrt{1+L^2}}\right]$,

\begin{equation}\label{Bound14}
|{\cal S}_1^{(p,q)}(P,z)|+r_0|\nabla{\cal S}_1^{(p,q)}(P,z)|\leq \frac{CE}{r_0}\left(\left(\frac{r}{r_0}\right)^{\beta}+\left(\omega_1\left(\frac{\epsilon}{E}\right)\right)^{\delta^{\eta}}\right),
\end{equation}
where $C$ only depends on the a priori data.
By a suitable choice of $r=r(\epsilon)$, we get
\begin{equation}\label{Bound16}
|{\cal S}_1^{(p,q)}(P,z)|+r_0|\nabla{\cal S}_1^{(p,q)}(P,z)|\leq  \frac{CE}{r_0}\widetilde{\omega}_2\left(\frac{\epsilon}{E}\right),
\end{equation}
where
\[
\widetilde{\omega}_2(t) =\left|\ln |\ln\left(t\right)|\right|^{\frac{-\gamma}{2A|\ln\delta|}}.
\]
By the transmission conditions
\[
{\cal S}_1^{(p,\cdot)}(\cdot,z)|_{D_0} = {\cal S}_1^{(p,\cdot)}(\cdot,z)|_{D_1}, \quad
\CC^0\nabla {\cal S}_1^{(p,\cdot)}(\cdot,z)n_1|_{D_0} =
\CC^1\nabla {\cal S}_1^{(p,\cdot)}(\cdot,z)n_1|_{D_1}, \quad \mbox{on }\Sigma_1,
\]
and by (\ref{Bound16}) we have, for every $p\in\{1,2,3\}$,
\begin{equation}\label{Bound17}
\|{\cal S}_1^{(p,\cdot)}(\cdot,z)|_{D_1}\|_{H^{\frac{1}{2}}(\Sigma_1')}+r_0\|\nabla{\cal S}_1^{(p,\cdot)}(\cdot,z)|_{D_1}\|_{H^{-\frac{1}{2}}(\Sigma_1')}\leq \frac{CE}{r_0}\tilde{\omega}_2\left(\frac{\epsilon}{E}\right),
\end{equation}
where $C$ only depends on the a priori data.

We can adapt the arguments in \cite{ARRV}  (see in particular Lemma 6.1 and Theorem 6.2) to obtain the following stability estimate for the Cauchy problem for ${\cal S}_1^{(p,\cdot)}(\cdot,z)$ in
$Q_{\frac{r_0}{3}L}(P_1)\cap D_1$
\begin{equation}\label{Bound18}
\|{\cal S}_1^{(p,q)}(\cdot,z)|\|_{L^{\infty}(B_{\overline{\rho}}(R_1))}\leq \frac{CE}{r_0}\left(\tilde{\omega}_2\left(\frac{\epsilon}{E}\right)\right)^{\xi},
\end{equation}
where $R_1=P_1-dn_1$, $d=
\frac{r_0}{3}\frac{L\sqrt{1+L^2}}{1+\sqrt{1+L^2}}$, $\overline{\rho} =\frac{r_0}{12}\frac{L}{1+\sqrt{1+L^2}}$, and the constants $\xi\in (0,1)$ and $C>0$ only depend on the a priori data.
Observe that by Proposition 5.5 in \cite{ARRV} there exists $h_1>0$, with $\frac{h_1}{r_0}$ depending only on $L$, such that
\[
(D_1)_h =\{x\in D_1\ |\ d(x,\partial D_1)>h\}\textrm{ is connected },\forall h\leq h_1,
\]
Let $ \overline{h}=\min\left\{h_1,\frac{r_0}{3}\frac{L\sqrt{1+L^2}}{1+\sqrt{1+L^2}}\right\}$. Then, $\frac{\overline{h}}{r_0}$ depends only on $L$, $\overline{(D_1)}_{\overline{h}}$ is connected and contains the points $R_1$ and $Q_2=P_2+\frac{8}{5}\overline{h}\sqrt{1+L^2}n_2$, with $P_2\in \Sigma_2$ as in \textbf{($A_1$)}. Let $\gamma$ be an arc contained in $\overline{(D_1)}_{\overline{h} }$ connecting $R_1$ with $Q_2$.
By iterating the three spheres
inequality \eqref{3s-2} first over a chain of balls with centers on $\gamma$ and then over a chain of balls of decreasing radius and contained in a suitable cone with vertex at $P_2$ and
axis in the direction $n_2$, we obtain
\[
|{\cal S}_1^{(p,q)}(y_r,z)|\leq \frac{CE}{\sqrt{rr_0}}\left(\tilde{\omega}_2\left(\frac{\epsilon}{E}\right)\right)^{\tau\delta^{\eta}},\quad\forall z\in K_0, \forall\ p,q\in\{1,2,3\},
\]
where $y_r=P_2+rn_2$, $\tau\in(0,1)$ and $C>0$ only depend on the a priori data. Repeating the estimating procedure for the function ${\cal S}_1^{(\cdot,q)}(y_r,\cdot)$  we obtain
\[
|{\cal S}_1^{(p,q)}(y_r,z_{\overline{r}})|\leq \frac{CE}{r}\left(\omega^*_2\left(\frac{\epsilon}{E}\right)\right)^{\tau\delta^{\eta +\overline{\eta}}},
\quad \forall\ p,q\in\{1,2,3\},
\]
where $C$ only depend on the a priori data and
\[
\omega^*_2(t)=\left|\ln\left|\ln\left|\ln\left(\frac{\epsilon}{E}\right)\right|\right|\right|^{\frac{-\gamma}{2A|\ln\delta|}}.
\]
Now, choosing a coordinate system centered at $P_2$, with $e_3=n_2$, and denoting by $\Gamma$ and $\overline{\Gamma}$ the Rongved solutions corresponding to the tensors $\CC_1\chi_{\mathbb{R}^3_+}+\CC_2\chi_{\mathbb{R}^3_-}$ and $\overline{\CC}_1\chi_{\mathbb{R}^3_+}+\overline{\CC}_2\chi_{\mathbb{R}^3_-}$, we get that
\[
|(\Gamma(e_3,ce_3)-\overline{\Gamma}(e_3,ce_3))e_3\cdot e_3|\leq CE \omega_2\left(\frac{\epsilon}{E}\right),
\]
where $C$ only depend on the a priori data and
\[
\omega_2(t)=\left|\ln\left|\ln\left|\ln\left|\ln\left(\frac{\epsilon}{E}\right)\right|\right|\right|\right|^{\frac{-\gamma}{2A|\ln\delta|}},
\]
so that, proceeding as in \cite{BFV}, we have
\[
\|\CC_2-\overline{\CC}_2\|\leq CE \omega_2\left(\frac{\epsilon}{E}\right).
\]
If $E=\|\CC_2-\overline{\CC}_2\|_\infty$, then
\[
\|\CC-\overline{\CC}\|_{\infty}=E\leq \frac{\epsilon}{\omega_2^{-1}(\frac{1}{C})}
\]
and the claim follows.
Otherwise, we proceed similarly iterating the procedure up to $k=j$ obtaining the desired result.
\end{proof}


\bibliographystyle{plain}

\begin{thebibliography}{99}

\bibitem[A]{Al1} G. Alessandrini, Stable determination of conductivity by boundary measurements, \textit{Applicable Analysis} 27, 153--172 (1988).

\bibitem[AdCMR]{AdCMR}G. Alessandrini, M. Di Cristo, A. Morassi, E. Rosset, Stable determination of an inclusion in an elastic body by boundary measurements,
 arXiv:1306.3349, 2014, to appear on \textit{SIAM J. Math. Anal.}

\bibitem[AM]{Al-M} G. Alessandrini, A. Morassi, Strong unique continuation for the
Lam\'{e} system of elasticity, \textit{ Comm. PDE}, 26,  1787--1810 (2001).

\bibitem[ARRV]{ARRV} G. Alessandrini, L. Rondi, E. Rosset,
S. Vessella, The stability for the Cauchy problem for elliptic
equations, \textit{ Inverse Problems} 25,  1--47 (2009).

\bibitem[AV]{AV} G.~Alessandrini, S.~Vessella, Lipschitz stability for the inverse conductivity problem,
\textit{Adv. in Appl. Math.}, 35, 207--241 (2005).

\bibitem[BV]{BaVe} V. Bacchelli, S. Vessella, Lipschitz stability for a stationary 2D inverse problem with unknown polygonal boundary, \textit{Inverse Problems} 22,  1627�1658 (2006).

\bibitem[BFV]{BFV} E. Beretta, E. Francini and S. Vessella,
\textit{Uniqueness and Lipschitz stability for the identification of Lam\'{e} parameters
{}from boundary measurements}, 2013, arXiv:1303.2443 [math.AP], to appear in Inverse Probl. Imaging.

\bibitem[BJK]{BJK} B. M. Brown, M. Jais, I. W. Knowles, A variational approach to an elastic inverse problem, \textit{Inverse Problems}, {\bf 21},  1953--1973 (2005)

\bibitem[ER]{ER} G. Eskin,  J. Ralston, On the inverse boundary value problem for linear isotropic elasticity, \textit{ Inverse
Problems}, 18, 907--922 (2002).

\bibitem[Ik]{Ik} M. Ikehata, Inversion formulas for the linearized problem for an inverse boundary value problem in elastic prospection, \textit{SIAM J. Appl. Math.} 50, 1635--1644 (1990).

\bibitem[IUY]{IUY} O.Y. Imanuvilov, G. Uhlmann and M.
Yamamoto, On uniqueness of Lam\'{e} coefficients {}from partial
Cauchy data in three dimensions, \textit{Inverse Problems}, 28,
Paper 125002 (2012).

\bibitem[LN]{LN}
Y.~Y.~Li, L.~Nirenberg, Estimates for elliptic systems {}from composite materials,
\textit{Comm. Pure Appl. Math.}, 56, 892--925 (2003).

\bibitem[Ma]{M} N. Mandache, Exponential instability in an inverse problem for the Schr\"{o}dinger equation,  \textit{Inverse Problems}, 17, 1435--1444 (2001).



\bibitem[N]{N} G. Nakamura, Inverse problems for elasticity, Selected papers on analysis and differential equations, 71--85, Amer. Math. Soc. Transl. Ser. 2, 211, Amer. Math. Soc., Providence, RI, 2003.


\bibitem[NU1]{NU1} G. Nakamura, G. Uhlmann, Identification of Lam\'e Parameters by Boundary Measurements, \textit{American Journal of Mathematics}, 115, 1161--1187 (1993).

\bibitem[NU2]{NU2} G. Nakamura, G. Uhlmann, Erratum: Global uniqueness for an inverse boundary value
problem arising in elasticity, \textit{Invent. Math.} 152,  205--207 (2003)(Erratum to Invent. Math. 118 (3) (1994) 457--474).

\bibitem[NU3]{NU3} G. Nakamura, G. Uhlmann, Inverse boundary problems at the boundary for an elastic system, \textit{SIAM J. Math. Anal.}, 26, 263--79 (1995).

\bibitem[R]{R} L. Rongved, Force interior to one of two joined semi-infinite solids, Proc. 2nd Midwestern
Conf. Solid Mech, 1--13, 1955.

\bibitem[XJY]{XJY} F. Xin, Z. Jing, F. Yingfang, Methods for identifying sub-regional material parameters of concrete dams using modal data, \textit{Acta Mechanica Solida Sinica}, 16, 88--94 (2003).


\end{thebibliography}

\end{document}